\documentclass[11pt,oneside]{amsart}

\usepackage{amsmath,ifthen, amsfonts, amssymb,
srcltx, amsopn, color, enumerate, hyperref}
\usepackage[cmtip,arrow]{xy}
\usepackage{pb-diagram, pb-xy}
\usepackage{overpic}

\usepackage[T1]{fontenc}

\newcommand{\showcomments}{yes}

\newsavebox{\commentbox}
{\ifthenelse{\equal{\showcomments}{yes}}%
{\footnotemark
        \begin{lrbox}{\commentbox}
        \begin{minipage}[t]{1.25in}\raggedright\sffamily\tiny
        \footnotemark[\arabic{footnote}]}
{\begin{lrbox}{\commentbox}}}%
{\ifthenelse{\equal{\showcomments}{yes}}%
{\end{minipage}\end{lrbox}\marginpar{\usebox{\commentbox}}}
{\end{lrbox}}}

\newcounter{intronum}
\newcounter{ax}
\newtheorem{thmi}{Theorem}

\theoremstyle{definition}

\newtheorem{remi}[intronum]{Remark}

\newtheorem{claim*}{Claim}

\DeclareMathOperator{\dimension}{dim}

\newcommand{\field}[1]{\mathbb{#1}}

\newcommand{\naturals}{\ensuremath{\field{N}}}
\newcommand{\reals}{\ensuremath{\field{R}}}
\newcommand{\Euclidean}{\ensuremath{\field{E}}}

\makeatletter

\newcommand{\Rmnum}[1]{\mathbf{{\expandafter\@slowromancap\romannumeral #1@}}}

\newcommand{\simp}{\ensuremath{\partial_{_{\triangle}}}}

\makeatother

\let\oldmarginpar\marginpar
\renewcommand\marginpar[1]{\-\oldmarginpar[\raggedleft\footnotesize #1]%
{\raggedright\footnotesize #1}}

\newcounter{enumitemp}

\newcommand{\cuco}[1]{\mathcal{#1}}

\begin{document}
\title{Corrigendum to ``The simplicial boundary of a CAT(0) cube complex''}
\author[M.F. Hagen]{Mark F. Hagen}
\address{Dept. of Pure Maths. and Math. Stat., University of Cambridge, Cambridge, UK}
\email{markfhagen@gmail.com}
\date{\today}
\maketitle

\begin{abstract}{We correct Theorem 3.10 of~\cite{Hagen:boundary} in the infinite-dimensional case.  No correction is needed in the finite-dimensional case.}
\end{abstract} 

In this note, we correct Theorem 3.10 of~\cite{Hagen:boundary} and record consequent adjustments to later statements.  In~\cite{Hagen:boundary}, we worked in a CAT(0) cube complex $\mathbf X$ with the property that there is no infinite collection of pairwise-crossing hyperplanes.  Such a complex may still contain cubes of arbitrarily large dimension.  It is in this situation that Theorem 3.10 requires correction; under the stronger hypothesis that $\dimension\mathbf X<\infty$, the theorem and its consequences hold as written in~\cite{Hagen:boundary}.  The extant results that use the simplicial boundary (see~\cite{chatterji2016median,durham2016boundaries,behrstock2012cubulated,hagen2014cocompactly,hagen2016hierarchical}) all concern finite-dimensional CAT(0) cube complexes and are thus unaffected.

\subsection*{Acknowledgment} I thank Elia Fioravanti for drawing my attention to the infinite-dimensional case and for reading a draft of this note.

\section*{Corrected statement and proof of Theorem 3.10}\label{sec:thm_proof}
Throughout, $\mathbf X$ denotes a CAT(0) cube complex which, according to standing hypotheses in~\cite{Hagen:boundary}, has countably many cubes and hyperplanes.  
The following replaces~\cite[Theorem 3.10]{Hagen:boundary}:

\begin{thmi}\label{thm:main}
Suppose that every collection of pairwise crossing hyperplanes in $\mathbf X$ is finite.  Let $v$ be an almost-equivalence class of UBSs.  Then $v$ has a representative of the form $\mathcal V=\bigsqcup_{i\in I}\mathcal U_i$, where $I$ is a finite or countably infinite set, each $\mathcal U_i$ is minimal, and for all distinct $i,j\in I$, if $H\in\mathcal U_j$, then $H$ crosses all but finitely many elements of $\mathcal U_i$, or the same holds with the roles of $i,j$ reversed.  Moreover, if $\dimension\mathbf X<\infty$, then the following hold:
\begin{itemize}
 \item $k=|I|\leq\dimension\mathbf X$;
 \item for all $1\leq i<j\leq k$, if $H\in\mathcal U_j$, then $H$ crosses all but finitely many elements of $\mathcal U_i$;
 \item if $\mathcal V'=\bigsqcup_{i\in I'}\mathcal U_i'$ is almost-equivalent to $\mathcal V$, and each $\mathcal U_i'$ is minimal, then $|I'|=k$ and, up to reordering, $\mathcal U_i$ and $\mathcal U_i'$ are almost equivalent for all $i$.   
\end{itemize}
\end{thmi}

\begin{remi}
Under the additional hypothesis that $\mathbf X$ is finite-dimensional, Theorem~\ref{thm:main} can be applied in~\cite{Hagen:boundary} everywhere that Theorem 3.10 is used; see Remark~\ref{rem:consequences} below.
\end{remi}

\begin{remi}\label{rem:example}
Consider the following wallspace.  The underlying set is $\reals^2$, and for each integer $n\geq0$, we have a set $\{H_i^n\}_{i\geq0}$ of walls so that:
\begin{itemize}
 \item $H_i^n$ is the image of an embedding $\reals\to\reals^2$ for all $i,n$;
 \item each $\{H_i^n\}_{i\geq0}$ has $H_i^n$ separating $H_{i+1}^n,H_{i-1}^n$ for all $i\geq1$;
 \item whenever $n<m$, we say $H^m_i$ crosses $H^n_j$ for $j>m$ and does not cross $H^n_j$ for $j\leq m$.
\end{itemize}
Let $\mathbf X$ be the dual cube complex, whose hyperplanes we identify with the corresponding walls.  Then each $\{H_i^n\}_{i\geq0}$ is a minimal UBS, and $\mathcal V=\bigsqcup_{n\geq0}\{H_i^n\}_{i\geq0}$ is a UBS satisfying the first conclusion.  However, $\mathcal V=\left(\bigsqcup_{n\geq0}\{H_i^n\}_{i>0}\right)\sqcup\mathcal V'$, where $\mathcal V'=\{H_1^n\}_{n\geq0}$ is a minimal UBS.  So, the finite dimension hypothesis is needed to obtain the ``uniqueness'' clause.
\end{remi}

\begin{proof}[Proof of Theorem~\ref{thm:main}]
Let $\mathcal V$ be a UBS representing $v$.  We first establish some general facts.

\textbf{Chains:}  A \emph{chain} in $\mathcal V$ is a set $\{U_i\}_{i=0}^\infty\subseteq\mathcal V$ of hyperplanes with the property that $U_i$ separates $U_{i+1}$ from $U_{i-1}$ for all $i\geq1$.  The chain $\{U_i\}_{i=0}^\infty$ is \emph{inextensible} in $\mathcal V$ if there does not exist $V\in\mathcal V$, one of whose associated halfspaces contains $U_i$ for all $i\geq0$.  Since $\mathbf X$ contains no infinite set of pairwise crossing hyperplanes, and $\mathcal V$ contains no facing triple, any infinite subset of $\mathcal V$ contains a chain.  Since $\mathcal V$ is unidirectional, any infinite $\mathcal W\subset\mathcal V$ contains a chain which is inextensible in $\mathcal W$.

\textbf{Almost-crossing:}  Let $\mathcal A,\mathcal B\subseteq\mathcal V$ be UBSs.  We write $\mathcal A\prec\mathcal B$ if $B$ crosses all but finitely many $A\in\mathcal A$, for all $B\in\mathcal B$.  Note that we can have $\mathcal A\prec\mathcal B$ and $\mathcal B\prec\mathcal A$ simultaneously; consider the hyperplanes in the standard cubulation of $\Euclidean^2$; in this case we say $\mathcal A,\mathcal B$ are \emph{tied}.  We have:
\begin{enumerate}[(i)]
 \item \label{item:subset} By definition, if $\mathcal A\subseteq\mathcal B$ and $\mathcal C$ is another UBS with $\mathcal B\prec\mathcal C$, then $\mathcal A\prec\mathcal C$.  Similarly, if $\mathcal A\subseteq\mathcal B$ and $\mathcal C\prec\mathcal B$, then $\mathcal C\prec\mathcal A$.
 \item \label{item:transitive} Suppose that $\mathcal A,\mathcal B,\mathcal C$ are minimal UBSs contained in the UBS $\mathcal V$.  Suppose that $\mathcal A\prec\mathcal B$ and $\mathcal B\prec\mathcal C$.  Then one of the following holds:
 \begin{itemize}
  \item $\mathcal A\prec\mathcal C$;
  \item $\mathcal A$ and $\mathcal B$ are tied and $\mathcal C\prec\mathcal A$.
 \end{itemize}
Indeed, this follows by considering $\mathcal A,\mathcal B,\mathcal C$ as the inseparable closures of chains (see below) and using that $\mathcal A\cup\mathcal B\cup\mathcal C\subset\mathcal V$ is unidirectional and contains no facing triple.  
\end{enumerate}

\textbf{Existence of the decomposition:}  The UBS $\mathcal V$ contains a minimal UBS $\mathcal U_1$.  Indeed, $\mathcal V$ contains a chain, being infinite, and contains the inseparable closure of the chain, being inseparable.  The proof of~\cite[Lemma~3.7]{Hagen:boundary} shows that this inseparable closure contains a minimal UBS.

Being infinite and unidirectional, $\mathcal U_1$ contains a chain $\mathcal C_1=\{U_i^1\}_{i=0}^\infty$ which is inextensible in $\mathcal U_1$.  By adding finitely many hyperplanes of $\mathcal V$ to $\mathcal U_1$, we can assume that $\mathcal C_1$ is inextensible in $\mathcal V$.  Moreover, the inseparable closure $\overline{\mathcal C_1}$ of $\mathcal C_1$ in $\mathcal V$ contains a UBS (as above) and is contained in $\mathcal U_1$, whence, by minimality of $\mathcal U_1$, we have $\mathcal U_1=\mathcal F_1\cup\overline{\mathcal C_1}$ for some finite $\mathcal F_1$.  We can remove $\mathcal F_1$ from $\mathcal U_1$ without affecting inseparability. Hence assume that $\mathcal U_1=\overline{\mathcal C_1}$.

Let $\mathcal V_1=\mathcal V-\mathcal U_1$.  If $\mathcal V_1$ is finite, then $\mathcal V$ is almost-equivalent to $\mathcal U_1$, and we are done, with $k=1$.  Hence suppose that $\mathcal V_1$ is infinite.  Note that $\mathcal V_1$ is unidirectional and has no facing triple.

Let $\mathcal V_1=\mathcal V_1^+\sqcup\mathcal V_1^-$, where $\mathcal V_1^+$ is the set of $V\in\mathcal V_1$ so that $V$ crosses all but finitely many elements of $\mathcal U_1$.  If $V\in\mathcal V_1^-$, then $V$ crosses $U_0$, for otherwise $V$ would form a facing triple with $U_j,U_{j'}$ for some $j,j'$.  Moreover, $V$ crosses only finitely many $U_j$.  Indeed, otherwise, $V$ crosses all but finitely many $U_j$, and hence crosses all but finitely many hyperplanes in $\overline{\mathcal C_1}=\mathcal U_1$, and we would have $V\in\mathcal V_1^+$.  

If $V,V'\in\mathcal V_1^-$, and $W$ separates $V,V'$, then $W$ crosses only finitely many elements of $\mathcal U_1$, and hence either $W\in\mathcal V_1^-$ or $W\in\mathcal U_1$.  Moreover, $W$ crosses $U_0$.  Now, since $\mathcal U_1=\overline{\mathcal C_1}$, no element of $\mathcal U_1$ crosses $U_0$, so $W\in\mathcal V_1^-$.  Hence $\mathcal V_1^-$ is inseparable.

On the other hand, suppose $V,V'\in\mathcal V_1^+$ and $W$ separates $V,V'$.  Then $W$ must cross all but finitely many elements of $\mathcal U_1$, so $W\in\mathcal V_1^+$ (the possibility that $W\in\mathcal U_1$ is ruled out since $\mathcal U_1$ is the inseparable closure of $\mathcal C_1$).  Hence $\mathcal V_1^+$ is inseparable.

Thus each of $\mathcal V_1^+$ or $\mathcal V_1^-$ is either finite or a UBS.  By definition, $\mathcal U_1\prec\mathcal V_1^+.$  We now check that, if $\mathcal V_1^-$ is infinite, then $\mathcal V_1^-\prec\mathcal U_1$.  Define a map $f:\mathcal V_1^-\to\naturals$ by declaring $f(V)$ to be the largest $j$ for which $V$ crosses $U_j$.  If $f^{-1}(j)$ is infinite for some $j$, then $f^{-1}(j)$ contains a chain $\mathcal D$, all of whose hyperplanes are separated by $U_j$ from $U_{j'}$ whenever $j'>j+1$, contradicting unidirectionality of $\mathcal V$.  Hence $f^{-1}(j)$ is finite for all $j$.  Together with the facts that each $V\in\mathcal V_1^-$ crosses $U_0$ and $\mathcal U_1=\overline{\mathcal C_1}$, this implies that $\mathcal V_1^-\prec\mathcal U_1$. In summary, one of the following holds:
\begin{enumerate}
 \item $\mathcal V=\mathcal U_1\sqcup\mathcal V^+_1\sqcup\mathcal V_1^-$, where $\mathcal U_1,\mathcal V_1^+$ are UBSs and $|\mathcal V_1^-|<\infty$.  In this case, $\mathcal U_1\cup\mathcal V_1^-$ is inseparable, so, by enlarging $\mathcal U_1$, we can write $\mathcal V=\mathcal U_1\sqcup\mathcal V^+_1$ where $\mathcal U_1\prec\mathcal V_1^+$.\label{item:U_1_min}
 \item $\mathcal V=\mathcal U_1\sqcup\mathcal V^+_1\sqcup\mathcal V_1^-$, where $\mathcal U_1,\mathcal V_1^-$ are UBSs and $|\mathcal V_1^+|<\infty$.  As above, we can, by enlarging $\mathcal U_1$ leaving $\mathcal V$ unchanged, write $\mathcal V=\mathcal U_1\sqcup\mathcal V_1^-$, with $\mathcal V_1^-\prec\mathcal U_1$.\label{item:U_1_max}
 \item $\mathcal V=\mathcal U_1\sqcup\mathcal V^+_1\sqcup\mathcal V_1^-$, and $\mathcal V_1^-\prec\mathcal U_1\prec\mathcal V_1^+$.\label{item:U_1_mid}
\end{enumerate}

We now apply the above construction of minimal sub-UBSs to $\mathcal V_1^+$ (in case~\eqref{item:U_1_min}), to $\mathcal V_1^-$ (in case~\eqref{item:U_1_mid}), or to both (in case~\eqref{item:U_1_max}).  Continuing in this way, using the above facts about $\prec$, we find a countable set $I$ and a subset $\bigsqcup_{i\in I}\mathcal U_i\subseteq\mathcal V$ so that each $\mathcal U_i$ is a minimal UBS, and $\mathcal U_i\prec\mathcal U_j$ or $\mathcal U_j\prec\mathcal U_i$ for all $i,j\in I$, and every chain in $\mathcal V$ is contained in $\bigsqcup_{i\in I}\mathcal U_i$.  Hence $\mathcal V-\mathcal U_i$ consists of those hyperplanes of $\mathcal V$ which do not lie in any chain.  Since any infinite set of hyperplanes in $\mathcal V$ contains a chain, there are finitely many such hyperplanes, so $\mathcal V$ is almost-equivalent to $\bigsqcup_{i\in I}\mathcal U_i$. 

\textbf{Uniqueness and dimension bound when $\dimension\mathbf X<\infty$:}  Observe that for any finite subset of $I$, $\mathbf X$ contains a set of pairwise crossing hyperplanes of the same cardinality, so $|I|=k\leq\dimension\cuco X<\infty$.  The uniqueness statement follows exactly as in the proof of~\cite[Theorem~3.10]{Hagen:boundary}; finite dimension is needed precisely because that argument uses that $k<\infty$.

\textbf{Ordering $I$ when $\dimension\mathbf X<\infty$:}  To complete the proof, it suffices to consider the UBS $\bigsqcup_{i=1}^k\mathcal U_i$, where each $\mathcal U_i$ is a minimal UBS and, for all $i,j$, either $\mathcal U_i\prec\mathcal U_j$ or $\mathcal U_j\prec\mathcal U_i$.  Let $\Gamma$ be the graph with a vertex for each $\mathcal U_i$, with a directed edge from $\mathcal U_i$ to $\mathcal U_j$ if $\mathcal U_i\prec\mathcal U_j$ but $\mathcal U_j\not\prec\mathcal U_i$.  By induction and the properties of $\prec$ established above, $\Gamma$ cannot contain a directed cycle, i.e. $\Gamma$ is a finite directed acyclic graph, whose vertices thus admit a linear order respecting the direction of edges.  Hence we can order (and relabel) the $\mathcal U_i$ so that $\mathcal U_i\prec\mathcal U_j$ when $i<j$.
\end{proof}

\begin{remi}\label{rem:consequences}
 The correction of Theorem~3.10 affects the rest of~\cite{Hagen:boundary} as follows:
\begin{itemize}
 \item Since UBSs are not used in Sections~1 or~2, none of the statements there is affected.
 \item In Section 3, Theorem~3.10 should be adjusted as above.  Following Theorem~3.10, one should then add the standing hypothesis that $\mathbf X$ is finite-dimensional.  The same standing hypothesis should be added in Section 4.
 \item In Section~5, all of the statements involving the simplicial boundary already hypothesize finite dimension, so no statement in that section is affected.
 \item In Sections 6.1 and 6.2, the hypothesis that $\mathbf X$ is strongly locally finite must be replaced everywhere by the hypothesis that $\mathbf X$ is locally finite and finite dimensional, because of the dependence on Theorem 3.14.  
\end{itemize}
\end{remi}

\begin{remi}[The infinite-dimensional case]\label{rem:inf_dim}
If one is interested in the infinite-dimensional, strongly locally finite case, many results in~\cite{Hagen:boundary} are still available with sufficient care.  For example, one can still define the simplicial boundary as the simplicial complex associated to the almost-containment partial ordering on the set of almost-equivalence classes of UBSs, but Remark~\ref{rem:example} shows that in this case, simplices of $\simp\mathbf X$ may contain ``more $0$--simplices than expected''.  Many statements, when rephrased in terms of UBSs rather than boundary simplices, still hold in this generality.  For example, the proof of Theorem~3.19 still shows that any UBS whose equivalence class is maximal in the above partial ordering determines a geodesic ray and Theorems 3.23 and 3.30 hold as written.  
\end{remi}

\bibliographystyle{alpha}
\bibliography{corrigendum}

\begin{thebibliography}{DHS16}

\bibitem[BH12]{behrstock2012cubulated}
Jason Behrstock and Mark~F Hagen.
\newblock Cubulated groups: thickness, relative hyperbolicity, and simplicial
  boundaries.
\newblock {\em arXiv preprint arXiv:1212.0182}, 2012.

\bibitem[CFI16]{chatterji2016median}
Indira Chatterji, Talia Fern{\'o}s, and Alessandra Iozzi.
\newblock The median class and superrigidity of actions on {CAT}(0) cube
  complexes.
\newblock {\em Journal of Topology}, 9(2):349--400, 2016.

\bibitem[DHS16]{durham2016boundaries}
Matthew~G Durham, Mark~F Hagen, and Alessandro Sisto.
\newblock Boundaries and automorphisms of hierarchically hyperbolic spaces.
\newblock {\em arXiv preprint arXiv:1604.01061}, 2016.

\bibitem[Hag13]{Hagen:boundary}
Mark~F Hagen.
\newblock The simplicial boundary of a {CAT}(0) cube complex.
\newblock {\em Alg. Geom. Topol.}, 13(3):1299--1367, 2013.

\bibitem[Hag14]{hagen2014cocompactly}
Mark~F Hagen.
\newblock Cocompactly cubulated crystallographic groups.
\newblock {\em Journal of the London Mathematical Society}, 90(1):140--166,
  2014.

\bibitem[HS16]{hagen2016hierarchical}
Mark~F Hagen and Tim Susse.
\newblock Hierarchical hyperbolicity of all cubical groups.
\newblock {\em arXiv preprint arXiv:1609.01313}, 2016.

\end{thebibliography}
\end{document}